\documentclass[11pt]{amsart}

\usepackage{amsgen,amsthm,amstext,amsbsy,amsopn,amsfonts,amssymb, enumerate}

\usepackage{amsmath}

\newcommand\la{\langle}
\newcommand\ra{\rangle}

\newcommand\hh{{\mathfrak h}}

\newcommand\nn{{\mathfrak n}}

\newcommand\sso{{\mathfrak{so}}}
\newcommand\vv{{\mathfrak v}}

\newcommand\ww{{\mathfrak w}}
\newcommand\zz{{\mathfrak z}}

\newcommand\CC{\mathbb C}

\newcommand\HH{\mathbb H}

\newcommand\OO{\mathbb O}

\newcommand\RR{\mathbb R}

\newcommand\so{{\mathrm{so}}}

\newcommand\ad{\operatorname{ad}}

\newcommand\id{\operatorname{Id}}

\renewcommand\Re{\operatorname{Re}}
\renewcommand\Im{\operatorname{Im}}

\theoremstyle{plain}

\newtheorem{thm}{Theorem}[section]
\newtheorem{lem}[thm]{Lemma}
\newtheorem{prop}[thm]{Proposition}
\newtheorem{cor}[thm]{Corollary}

\theoremstyle{definition}
\newtheorem{defn}[thm]{Definition}
\newtheorem{rem}[thm]{Remark}
\newtheorem{example}[thm]{Example}

\begin{document}

\title[Magnetic fields on non-singular 2-step nilpotent Lie groups]
{Magnetic fields on non-singular 2-step nilpotent Lie groups}

\author{Gabriela P. Ovando, Mauro Subils}

\thanks{{\it (2000) Mathematics Subject Classification}: 70G45, 22E25, 53C99, 70G65 }

\thanks{{\it Key words and phrases}:  2-step nilmanifolds, closed 2-forms, $H$-type Lie groups.
 }

\thanks{Partially supported by  SCyT (UNR)}

\address{ Departamento de Matem\'atica, ECEN - FCEIA, Universidad Nacional de Rosario.   Pellegrini 250, 2000 Rosario, Santa Fe, Argentina.}

\

\email{gabriela@fceia.unr.edu.ar}

\email{subils@fceia.unr.edu.ar}


\begin{abstract}  The aim of this work is the study of left-invariant   magnetic fields on 2-step nilpotent Lie groups.   While the existence of closed 2-forms for which the center is either nondegenerate or in the kernel of the 2-form, is always guaranteed, the existence of closed 2-forms for which the center is isotropic but not in the kernel of the 2-form, is a special situation. These 2-forms are called of type II. We obtain a strong obstruction for the existence on non-singular Lie algebras. Moreover, we prove that the only $H$-type Lie groups admitting closed 2-forms of type II are the real, complex and quaternionic Heisenberg Lie groups of dimension three, six and seven, respectively.  We also  prove the non-existence of uniform magnetic fields under certain hypotheses. Finally we give a construction of non-singular Lie algebras, proving that in some families of these examples there are no closed 2-form of type II.
\end{abstract}

\maketitle

 \noindent\section{Introduction}

 Recently, in \cite{AF} it was proved that nilpotent fat sub-symmetric spaces  are always 2-step nilpotent groups. 
 Fat bundles were introduce by Weinstein in \cite{We} in order to study distributions which are in some sense ``completely non flat''.  On this bundles any vector field on the distribution generates all the tangent bundle by taking brackets with other horizontal vector fields, which is equivalent to say that the symbol of this distribution is a non-singular 2-step nilpotent Lie algebra as introduced in \cite{Eb}. So they are also called fat Lie algebras \cite{M} \cite{LO}.
 
 The study  of 2-forms on 2-step nilpotent Lie groups was treated by many authors in different frameworks, obtaining important results in the following topics: Killing Yano 2-forms \cite{BDS},  Killing forms \cite{dBM},  symplectic structures (see for instance \cite{dB1,dB2,DT,Mi}). In the present paper we focus on closed left-invariant 2-forms  on a 2-step nilpotent Lie group $N$. 
 These 2-forms are called {\em magnetic fields}, since they are  deeply related to  magnetic trajectories (see for instance \cite{Su,OS} and references herein). In fact, magnetic trajectories are curves on a Riemannian manifold $(M, \la \,, \, \ra)$,  solutions of the equation \begin{equation}\label{magneteq}
 \nabla_{\gamma'} {\gamma'} = F{\gamma'},
 \end{equation}
 where $\nabla$ is the Levi-Civita connection on $M$ and $F$ is a (1,1)-tensor giving rise to a closed 2-form $\omega=\la F\cdot, \cdot \ra$, which could be degenerate. It is important to notice that different magnetic fields produce magnetic trajectories with strong differences with respect to their behaviour, specially when considering magnetic fields of type II. 
 
  Whenever  $\Gamma$ is a discrete co-compact subgroup of the Lie group $N$, the quotient space $M=\Gamma \backslash N$ is a compact nilmanifold. By Nomizu's Theorem   every de Rahm cohomology group $H^i(M, \RR)$ is isomorphic to $H^i(\nn)$. In particular one looks for closed 2-forms on the  on the Lie algebra of $N$, namely $\nn$. 

 	Dotti and Tirao in \cite{DT} proved that if the 2-step nilpotent Lie group $N$ admits a left-invariant symplectic structure, then 
 $$2\dim[\nn, \nn] \leq \dim \nn +  1.$$
 By studying the existence of closed 2-forms we proved that 
 if $\nn$ is non-singular and
  $$\dim \nn > 3 \dim [\nn, \nn],$$
  then any closed 2-form $\omega$ on $\nn$ satisfies
  $$\omega(Z,U)=0, \qquad \mbox{ for all } Z\in \zz, U\in \nn.$$
 
 This is equivalent to say that for any metric, the corresponding skew-symmetric map induced by $\omega$ preserves the decomposition 
 $$(*) \qquad \nn =\vv \oplus \zz, \qquad \mbox{ for } \vv=\zz^{\perp}, $$
 where $\zz$ denotes the center of $\nn$. Notice that in the non-singular case, the commutator $[\nn, \nn]$ coincides with the center of the Lie algebra $\zz$. 
 
 We proved the following theorem:
 
 \smallskip
 
 {\it Theorem  (see Theorem \ref{thm2} )} Let $N$ be an $H$-type Lie group. It admits a left-invariant Lorentz force, such that $F\zz \subseteq \vv$ and $F\vv \subseteq \zz$ if and only if the Lie algebra $\nn$ is (isomorphic to) the Heisenberg Lie algebra of dimension three, the complex Heisenberg Lie algebra of dimension six or the quaternionic Heisenberg Lie algebra of dimension seven.
 
 \smallskip
  
  These are exactly the Lie algebras such that the corresponding nilmanifolds $T^t\times \Gamma \backslash N$ admit a symplectic structure, where $t=0,1$, \cite{DT}. The proof given there makes use of cohomological methods.
  In other words, these results prove the following equivalences for  an $H$-type Lie group $N$:
  \begin{itemize}
  	\item  $T^t \times N$  admits a (left-invariant) symplectic structure, $t=0,1$.
  	\item $N$ admits a closed left-invariant 2-form such that  $\omega(\zz, \zz)=0$ but $\omega(\zz, \nn)\neq 0$.
  	\item the Lie algebra of $N$ is isomorphic to the Heisenberg Lie algebra of dimension three $\hh_3$, the complex Heisenberg Lie algebra of dimension six $\hh(\CC)$ or the quaternionic Heisenberg Lie algebra of dimension seven $\hh(\HH)$. 
  \end{itemize}

Indeed, along the text we show examples of singular and almost non-singular Lie algebras admitting closed 2-forms of type II. To obtain these results we consider the decomposition above in (*) and decompose a Lorentz map $F:\nn \to \nn$ as $F=F_1+F_2$, where $F_1$ preserves the decomposition, while $F_2$ verifies $F_2(\vv)\subset \zz$ and $F_2(\zz)\subset \vv$. In this way the closeness condition can be written in terms of conditions for each one of these maps $F_i$, that we deepen. By considering the non-singular property, one gets the first lemma and after that, a deeply study on $H$-type Lie algebras, gives the proof of the theorem, using features of the corresponding division algebras. 

A later question was considered: are there other non-singular Lie algebras than the $H$-type ones, admitting closed 2-forms of type II? At this point one can realized that the most known examples of non-singular Lie algebras are the $H$-type ones and there are few methods to construct others. For center of dimension $2$ they are classified in \cite{LT}, where they also define the concept of $\tilde{H}$-type Lie algebras. Other interesting examples with center of dimension $3$ are given in \cite{LO}. But all these examples of non $H$-type Lie algebras appear on dimensions of non irreducible $H$-type Lie algebras. Inspired by \cite{LT}, we propose another way to construct non-singular Lie algebras, by modifying one of the Lie bracket relations on an $H$-type Lie algebra. In this way we get non-singular Lie algebras that are not of type $H$ for any admissible dimension greater than $7$.  

In some families of these new examples of non-singular Lie algebras, we proved that there are no closed 2-forms of type II. Thus, the question of existence of closed 2-forms of type II in non-singular Lie algebras is still open.  


\section{Lie groups of step two with a left-invariant metric}\label{general}

A Lie group is called 2-step nilpotent if its Lie algebra is 2-step nilpotent, that is, the Lie bracket satisfies $[[U,V], W]=0$ for all $U,V,W\in \nn$. Throughout this paper Lie groups, so as their Lie algebras are considered over $\RR$. 



\begin{example} \label{exa1} The smallest dimensional 2-step nilpotent Lie group is the Heisenberg Lie group $H_3$. It has dimension three and its Lie algebra is spanned by vectors $e_1, e_2, e_3$ satisfying the non-trivial Lie bracket relation
	$$[e_1,e_2]=e_3.$$
	The Lie group $H_3$ can be modelled on $\mathbb R^3$ equipped with the product operation given by
	$$(v_1,z_1)(v_2,z_2)=(v_1+v_2, z_1+z_2+\frac{1}2 v_1^tJv_2),$$
	where $v_i=(x_i,y_i)$, i=1,2 and $J:\RR^2 \to \RR^2$ is the linear map $J(x,y)=(y, -x)$. By using this, usual computations show that a basis of left-invariant vector fields is given at $p=(x,y,z)$ by
	$$e_1(p)=\partial_x -\frac12 y \partial_z, \quad e_2(p)=\partial_y+\frac12 x \partial_z, \quad e_3(p)=\partial_z.$$
Another presentation of the Heisenberg Lie group is given by $3\times3$-triangular real  matrices with 1's on the diagonal with the usual multiplication of matrices.  
	\end{example} 

A Riemannian  metric $\la\,,\,\ra$ on the Lie group  $N$ is called {\it left-invariant} if  translations on the left by elements of the group are isometries. Thus, a left-invariant metric  is determined at the Lie algebra level $\nn$, usually identified with the tangent space at the identity element $T_eN$. The metric on the Lie algebra, also denoted $\la\,,\,\ra$, determines an orthogonal decomposition as vector spaces: 
\begin{equation}\label{decomp2}
	\nn=\vv \oplus \zz, \quad \mbox{ where }\quad \vv =\zz^{\perp}
\end{equation}
and $\zz$ denotes the center of $\nn$. The subspaces $\vv$ and $\zz$ induce left-invariant distributions on $N$, denoted by $\mathcal V$ and $\mathcal Z$.  

The decomposition in Equation \eqref{decomp2} induces the skew-symmetric maps $j_Z:\vv \to \vv$, for every $Z\in \zz$,  implicitly defined by 
\begin{equation}\label{j}
	\la Z, [V,W]\ra =\la j_Z V, W \ra \qquad \mbox{ for all } Z\in \zz, V, W\in \vv. 
\end{equation}

Note that $j:\zz \to \mathfrak{so}(\vv)$ is a linear map. Let $C(\nn)$ denote the commutator of the Lie algebra $\nn$. One has the splitting 
$$\zz=C(\nn)\oplus \ker(j)$$ as orthogonal direct sum of vector spaces. In fact, 
\begin{itemize}
	\item since $\la Z, [U,V]\ra=0$ for all $U,V\in \vv$ and $Z\in \ker(j)$,  then $\ker(j)\perp C(\nn)$.
	\item $\dim \zz = \dim \ker(j)+ \dim C(\nn)$. 
	\item The restriction $j:C(\nn) \quad \mapsto \quad \mathfrak{so}(\vv)\quad \mbox{is injective}.$
\end{itemize} 
In fact, assume $j_Z=j_{\bar{Z}}$ for $Z, \bar{Z}\in C(\nn)$. Then $j_{Z-\bar{Z}}=0$, so that $Z-\bar{Z}\in \ker(j)\cap C(\nn)$. Thus $Z-\bar{Z}=0$. 

See the proof of the next result in Proposition 2.7 of \cite{Eb}. 
 
 \begin{prop} \cite{Eb} Let $(N, \la,,\,\ra)$ denote a 2-step nilpotent Lie group with a left-invariant metric. Then 
 	\begin{itemize}
 		\item the subspaces $\ker j$ and $C(\nn)$ are commuting ideals in $\nn$. 
 		\item Let $E=\exp(\ker(j))$. Then $E$ is the Euclidean de Rham factor of $N$ and $N$ is isometric to the Riemannian product of the totally geodesic submanifolds $E$ and $\bar{N}$ where $\bar{N}=\exp(\vv \oplus C(\nn))$. 
 	\end{itemize}
 	
 	\end{prop}

 \begin{example}
 	Let $\hh_3$ denote the Heisenberg Lie algebra of dimension three with basis $e_1, e_2, e_3$ as in Example \ref{exa1}. Take the metric so that  this basis is orthonormal. It is not hard to see that the center is the subspace
 	 $\zz=span\{e_3\}$, while its orthogonal complement is  the subspace $\vv=span\{e_1, e_2\}$ and moreover the map $j:\zz\to \so(\vv)$ is generated by 
 	$$j_{e_3}=\left( \begin{matrix}
 		0 & -1\\
 		1 & 0 
 	\end{matrix}\right), 
 $$
 in the basis $e_1, e_2$ of $\vv$. 
 \end{example}

\begin{defn}
A  2 -step nilpotent real Lie algebra $\nn$ with center $\zz$ is called {\em non-singular}  if  $\ad(X): \mathfrak{n} \rightarrow \mathfrak{z}$ is onto for any $X \notin \mathfrak{z}$ \cite{Eb}. The corresponding 2-step nilpotent Lie group will be called non-singular. 

\end{defn}

See the next examples of non-singular Lie algebras. 

\begin{example} {\it Heisenberg Lie algebras.} \label{ExHeis} Let $n\geq 1$ be any integer and let $X_1,Y_1, X_2, Y_2,$ $ \hdots, X_n, Y_n$ be any basis of a real vector space $\vv$ isomorphic to $\RR^{2n}$. Let $Z$ be an element generating a one dimensional space $\zz$. Define a Lie bracket by $[X_i, Y_i]=-[Y_i, X_i]=Z$ and the other Lie brackets by zero. The Lie algebra $\hh_{2n+1}=\vv \oplus \zz$ is the $(2n+1)$-dimensional Heisenberg Lie algebra.   	
\end{example}

\begin{example} {\it Quaternionic  Heisenberg  Lie algebras.} Let $n\geq 1$ be any integer. For each integer $1\leq i \leq n$,  let  $\mathbb H_i$ be a  four dimensional real vector spaces with basis $X_i,Y_i, V_i, W_i$. Let $\zz$ denote a three dimensional real vector space with basis $Z_1, Z_2, Z_3$. Consider the vector space direct sum $\nn=\vv \oplus \zz$, where $\vv=\bigoplus_i \mathbb H_i$. Define a Lie bracket on the Lie algebra $\nn$, $[\,,\,]$, that is $\RR$-bilinear and skew-symmetric with non-trivial relations as follows:
	
	\smallskip
	
	$[Z, \xi]=0$ for all $Z\in \zz, \xi\in \nn$, 
	$$	
	\begin{array}{llrl} 
		[X_i,Y_i]=Z_1, & [X_i,V_i]=Z_2, & [X_i, W_i]=Z_3 & \mbox{ for } 	1\leq i \leq n, \\
		{ [V_i, W_i]=Z_1}, & [Y_i, W_i]=-Z_2, & [Y_i, V_i]=Z_3 & \mbox{ for } 	1\leq i \leq n. 	
	\end{array}
	$$
	The resulting Lie algebra is called the quaternionic Heisenberg Lie algebra of dimension $4n+3$. 
\end{example}

Once the 2-step nilpotent Lie algebra $\nn$ is equipped  with a metric, one has the corresponding maps $j_Z:\vv \to \vv$ defined in Equation \eqref{j}. The non-singularity property is equivalent to the condition that  any map $j_Z:\vv\to\vv$ is non-singular for every $Z\in\zz$. And this condition is independent of the metric.

 Non-singular Lie algebras are also known as {\em fat} algebras because they are the symbol of distributions associated to fat bundles (see \cite{M} and \cite{We}).     

More generally, a 2-step nilpotent Lie algebra $\nn$ is said
\begin{itemize}
	\item {\em almost non-singular} if there are elements $Z, \widetilde{Z}\in \zz$ such that $j_Z$ is non-singular but $j_{\widetilde{Z}}$ is singular, 
	\item  {\em singular} if any map $j_Z$ is singular for every $Z\in \zz$.
\end{itemize} 

A family of non-singular Lie algebras is provided by $H$-type Lie algebras, which are defined as follows. 

Let $(\nn, \la\,,\,\ra)$ denote a 2-step nilpotent Lie algebra equipped with a metric. If the map $j_Z:\vv\to\vv$ is orthogonal for every $Z\in\zz$ with $\la Z, Z \ra =1$,  then the Lie algebra $\nn$ is a {\em Lie algebra  of type H  } \cite{K} (also known as $H$-type Lie algebras). Equivalently, the 2-step nilpotent Lie algebra $\nn$ is of type H if and only if   $$j_Z^2=-\la Z, Z\ra Id, \qquad \mbox{ for every }Z\in\zz, $$
which is equivalent to 
$j_Z j_{\widetilde{Z}}+ j_{\widetilde{Z}}j_Z =-2\la Z, \widetilde{Z}\ra Id$, for $Z, \widetilde{Z}\in \zz$. 
By making use of this identity one can prove that 
$$[X,j_ZX]=\la X, X\ra Z$$
for every $X\in\vv$ and $Z\in\zz$.
	The real, complex and quaternionic Heisenberg algebras are examples of $H$-type Lie algebras.

In a general $H$-type Lie algebra $\nn=\vv\oplus\zz$ with $\dim\,\zz=m$, the real vector space $\vv$ is a module over the Clifford algebra $\mathrm{Cl}(m)$ associated to the quadratic form $-\|\, . \,\|^2$. So the existence problem of Lie algebras of Type $H$ is naturally associated to the existence of representations of Clifford algebra $\mathrm{Cl}(m)$ (see \cite{K2}).

Let $n\geq 1$ be any integer  and write $n=(2a + 1)2^b$ with  $b =c + 4d$, where $a$, $b$, $c$ and $d$ are integers and $0\leq c\leq 3$. Define
$\rho(n)=2^c+8d$, called the {\em Radon-Hurwitz numbers}. They were introduced to calculate the number of linearly independent vector fields on the unit sphere on $\RR^{n}$: $\mathbb S^{n-1}=\{x\in \RR^{n} \,:\, \|x\|=1\}$. 

\begin{prop}\label{pairs}\cite{K,M}
	Let $n$,$m$ be non-negative integers. The following statements are equivalent. 
	\begin{enumerate}
		\item There exist a non-singular Lie algebra $\nn=\vv\oplus\zz$ with $\dim\,\vv=n$ and $dim\,\zz=m$.
		\item There exist an $H$-type Lie algebra $\nn=\vv\oplus\zz$ with $\dim\,\vv=n$ and $\dim\,\zz=m$.
		\item $\mathbb{S}^{n-1}$ admits $m$ linearly independent vector fields.
		\item There exist a $\mathrm{Cl}(m)$-module of dimension $n$.
		\item $m<\rho(n)$.
	\end{enumerate}
\end{prop}

Moreover the next two theorems give a complete classification of $H$-type Lie algebras, see for instance \cite{CDKR,K2,LT}.

\begin{thm}
	Every module over the Clifford algebra $\mathrm{Cl}(m)$ is the 
	direct sum of irreducible modules. Up to isomorphism, there is precisely one irreducible module $\vv_{m}$ over $\mathrm{Cl}(m)$ for $m \not\equiv 3\,(mod\, 4)$, and there are precisely two $\vv_{m}^{+}$ and $\vv_{m}^{-}$ 
	for $m \equiv 3\, (mod\, 4)$. 
\end{thm}

As the next theorem proves, in many cases the $H$-type Lie algebra is determined by the pair $(\dim \vv,\dim\zz)$.

\begin{thm}
	Up to isomorphism, any $H$-type Lie algebra $\nn$ with center $\zz$ of dimension $m$ is of the form:
	\begin{enumerate}
		\item $\nn=(\vv_{m})^k\oplus\zz$ for $m \not\equiv 3\,(mod\, 4)$,
		\item $\nn=(\vv_{m}^{+})^p\oplus(\vv_{m}^{-})^q\oplus\zz$ for $m \equiv 3\,(mod\, 4)$.
	\end{enumerate} 
	Where two pairs of exponents $p, q$ and $r, s$ in $(b)$ give isomorphic algebras if and only if 
	$\{p, q\} = \{r, s\}$.
\end{thm}

Whenever  $k=1$ or $p+q=1$ we say that the corresponding $H$-type Lie algebra is {\it irreducible}. 

\subsection{$H$-type Lie algebras from normed real division algebras}\label{extypeH}

	Next we construct some important examples of $H$-type Lie algebras that arise from normed real division algebras. Firstly,  recall some basic notions about these division algebras. See \cite{KS} for further references.
	
	A normed real division algebra  $\mathbb{A}$ is a vector space with a bilinear product (non necessary associative) that have identity and in which every non-zero element has an inverse. Also they are endowed with a metric such that
	$$\|a.b\|=\|a\|\|b\|, \quad \mbox{ for all  } a,b\in \mathbb A.$$  
	Hurwitz's Theorem asserts that the only normed real division algebra are the real numbers: $\RR$,  the complex numbers: $\CC$,  the quaternions: $\HH$ or the octonions: $\OO$.

	The subspace given by $\Im(\mathbb{A})=\mathbb{R}^{\perp}$ is the subspace of purely imaginary vectors and every $V\in\mathbb{A}$ admits the decomposition $V=\Re(V)+\Im(V)$ where  $\Re(V)\in\mathbb{R}$ and $\Im(V)\in \Im(\mathbb{A})$, respectively the real and imaginary parts of the element $V\in \mathbb A$. So the conjugation mapping $V=\Re(V)+\Im(V)\to\overline{V}=\Re(V)-\Im(V)$ is the reflection across $\mathbb{R}$. Indeed, in the case $\mathbb A=\RR$, one has no imaginary subspace. 
	
	The metric on $\mathbb{A}$ can be described by means of the next relation:
	$$\la V,\, W\ra = \Re(V.\overline{W}).$$ 
The following useful identities hold:
	\begin{enumerate}
		\item $V.\overline{V}=\|V\|^2$  for every  $V\in\mathbb{A}$,
		\item $\la V,\, \Im(W)\ra = \la V,\, W\ra$ for all  $V\in\Im(\mathbb{A})$ and $W\in\mathbb{A}$,
		\item $\Re(V.W)=\Re(W.V)$  for any $V,\,W\in\mathbb{A}$,
		\item $\Re((V.W).U)=\Re(V.(W.U))$  for all $V,\,W,\,U\in\mathbb{A}$,
		\item $\la V.U,\, W\ra = \la V,\, W.\overline{U}\ra$ and  $\la U.V,\, W\ra = \la V,\, \overline{U}.W\ra$ for all  $V,\,W,\,U\in\mathbb{A}$,
	\end{enumerate}
	
	It is important to remark that both division algebras $\mathbb{H}$ and $\mathbb{O}$ are non-commutative and their center is $\mathbb{R}$. Moreover the algebra of octonions $\mathbb{O}$ is also non-associative, but it is alternative; which  means that the subalgebra generated by any two elements is associative. Furthermore, every non-zero element in $\mathbb{O}$ is contained in a subalgebra isomorphic to $\mathbb{H}$.

	We have two type of constructions of $H$-type Lie algebras associated to a normed real division algebra. 
	
	{\em Construction (i).} Assume that $\mathbb{A}=\mathbb{C}$, $\mathbb{H}$ or $\mathbb{O}$. 
	
	Take the direct sum vector space $\vv \oplus \zz$, where the subspaces $\vv$ and $\zz$ are respectively given by $\vv=\mathbb{A}$ and $\zz=\Im(\mathbb{A})$, each one equipped with its natural inner product from $\mathbb A$, such that $\zz$ is orthogonal to $\vv$.
	
	 Define the Lie bracket $[\cdot ,\cdot]:\vv\times\vv\to\zz$ as
	$$[U,V]=-\Im(U.\overline{V}) \quad \mbox{ and } \quad [Z, U]=0.$$
	Thus for any  $Z\in\zz$, and $U,V\in\vv$ one gets,
	$$\la j_Z U,V\ra = \la Z, [U,V]\ra = \la Z,-\Im(U.\overline{V})\ra=\la Z,\Im(V.\overline{U})\ra = \la Z,V.\overline{U}\ra = \la Z.U,V\ra.$$  
	This implies
	$$j_Z(V)=Z.V\quad \mbox{ for every } Z\in\zz, $$ 
	from which one obtains that 
	$$j_Z^2=-\|Z\|^2 \, \id, $$
	condition that derives from the fact that  $Z$ is a purely imaginary element in $\mathbb{A}$. Notice that $\langle Z, Z \rangle=\|Z\|^2$, for $Z\in\zz=\Im(\mathbb A)$. This is the family of $H$-type Lie algebras with pair $(n, m)$ where $m=\dim \zz = \dim \Im(A)$ and $n=\dim\vv=\dim A$. In this way, one gets the Heisenberg Lie algebra of dimension $3$ associated  with the pair $(2,1)$, the quaternionic  Heisenberg  Lie algebra of dimension $7$  with pair $(4,3)$ and the octonionic  Heisenberg  Lie algebra of dimension $15$ with pair $(8,7)$. 
	
	Observe that in this construction we have to be careful because we are identifying $\Im(\mathbb{A})$  with $\zz$ and with a subspace of $\vv$. For example, if $U=1\in\vv$ and $Z\in\zz$ then $J_{Z} U=1.Z=Z\in\vv$ (the same element of $\Im(\mathbb{A})$). Also, if $U,\, V\in\vv$ are purely imaginary and orthogonal then $[U,V]=U.V\in\zz$. 
	
	\smallskip
	
{\em Construction (ii).} Take the real vector spaces $\vv$ and $\zz$ as follows: $\vv=\mathbb{A}\times\mathbb{A}$ and $\zz=\mathbb{A}$. As above take the metric on $\vv$ and $\zz$ induced from that one on $\mathbb A$ (the usual product metric on $\vv$) and make $\zz$ orthogonal to $\vv$ in the direct sum $\vv\oplus \zz$.

 Define the Lie bracket on $\vv\oplus\zz$, in such way that 
\begin{itemize}
	\item $[Z, X]=0$ for all $Z\in \zz$, $X\in \vv\oplus \zz$, and 
	\item $[\, ,\, ]:\vv\times\vv\to\zz$ is given by
	$$[(U,V),(\widetilde{U},\widetilde{V})]=U.\widetilde{V}-\widetilde{U}.V, \quad\mbox{ for all } U,V,\widetilde{U}, \widetilde{V} \in \mathbb A.$$
\end{itemize}

Thus, from the properties  above, and for $Z\in\zz$, $U,V\in\vv$ one gets, 
	\begin{align*}
		\la j_Z(U,V),(\widetilde{U},\widetilde{V})\ra = \la Z, U.\widetilde{V}-\widetilde{U}.V\ra = \la Z, U.\widetilde{V}\ra - \la Z, \widetilde{U}.V\ra\\=
		\la \overline{U}.Z,\widetilde{V}\ra - \la Z.\overline{V}, \widetilde{U}\ra = \la (-Z.\overline{V},\overline{U}.Z), (\widetilde{U}, \widetilde{V})\ra.  
	\end{align*} 
	This implies the following
	 $$j_Z(U,V)=(-Z.\overline{V},\overline{U}.Z),$$
	from which one derives that
	
	  $j_Z^2(U,V)=(-Z.\overline{Z} U,-V.\overline{Z}.Z)= -\|Z\|^2(U,V)$ for all $U,V\in \mathbb A$. 
	  
	  These Lie algebras correspond to  the complex Heisenberg Lie algebra of dimension $6$  with pair $(4,2)$, and the $H$-type Lie algebras of dimension $12$ with pair $(8,4)$ and that of dimension $24$ with pair $(16,8)$. 
	  
	  \begin{example}
	  	Take the subset of $\vv=\mathbb A \times \mathbb A$ given by $\mathbb A \times \{0\}$. By the formulas above, it is clear that this is an abelian subalgebra. The same holds for the set $\{0\}\times \mathbb A\subset \vv$. 
	  \end{example}
	
	\begin{lem}\label{maxabelianHO} Let $\nn=\vv \oplus \zz$ denote a Lie algebra of type H as given in Construction (ii), for the division algebras of the  quaternions and octonions. The only maximal abelian subalgebras in $\vv$ of dimension $s=\dim \mathbb{A}$ are
	 the subspaces given by $\{0\}\times\mathbb{A}$ and $\mathbb{A}\times\{0\}$. 
\end{lem}

\begin{proof}
Let $\ww$ be an abelian subalgebra of $\vv=\mathbb{A}\times\mathbb{A}$ with $\dim\ww = \dim\,\mathbb{A}$. Let $\pi_1:\ww \to \mathbb{A}$ denote the projection to the first component, while $\pi_2:\ww \to \mathbb{A}$ denote the projection to the second one.

Let $X,Y\in \ww$ such that  $\pi_1(X)=\pi_1(Y)$. Thus, $X$ and $Y$ have the same first component $U$. Now, notice that  for any element $U\in \mathbb{A}$, with  $U\neq 0$, and $V,W\in \mathbb{A}$ it holds:
\begin{align*}
	[(U,V),(U,W)]=0\Leftrightarrow U.W-U.V=0\Leftrightarrow U.(W-V)=0 \Leftrightarrow W=V.
\end{align*}
This implies that $\pi_1$ is bijective if it is nontrivial. And with the same reasoning one also gets that the projection onto the second component $\pi_2:\ww \to A$ is either trivial or bijective. 

Assume that both projections  $\pi_1$ and $\pi_2$ are bijective maps. Then, there exist elements $U,\,\widetilde{U}\in\mathbb{A}$ such that the pairs  $(U,1),\,(1,\widetilde{U})$ belong to the subalgebra $\ww$. 

 Since it holds
$$0=[(U,1),(1,\widetilde{U})]=U.\widetilde{U}-1,$$
one derives that $\widetilde{U}$ is the inverse of $U$,  $\widetilde{U}=U^{-1}$.
Now, for any  $V\in \mathbb{A}$, there exists $\widetilde{V}\in \mathbb{A}$ such that $(\widetilde{V},V)\in\ww$. By taking the Lie bracket one gets 
$$[(U,1),(\widetilde{V},V)]=[(1,U^{-1}),(\widetilde{V},V)]=0,$$
which gives $U.V-\widetilde{V}=V - \widetilde{V}U^{-1}=0$. Thus we obtain
$$U.V=\widetilde{V}=V.U$$
Since $V$ is arbitrary, the relation above says that $U$ is in the center of $\mathbb{A}$ i.e. $U\in\mathbb{R}$. From our computations, all elements of $\ww$ are of the form $(U.V,V)$ for $V\in\mathbb{A}$. But for any $V,W\in \mathbb A$ and $U\neq 0$, we have
$$0=[(U.V,V),(U.W,W)]=(U.V).W-(U.W).V=U.(V.W-W.V),$$
where the associativity holds since $U\in \RR$. Thus, $V.W-W.V=0$, which says that  all elements in $\mathbb{A}$ must commute, which is a contradiction.
Thus one of the projections $\pi_1$ or $\pi_2$ must be trivial. 
\end{proof}

 All Lie algebras obtained either by Construction  (i)  or (ii) are irreducible. The other irreducible $H$-type Lie algebras with $\dim\zz \leq 8$ have pairs  $(8,5)$ and $(8,6)$. These can be derived  from the $(8,7)$ pair by using the following general construction.   

\

 {\em Construction (iii).} Let $\nn=\vv\oplus\zz$ be the orthogonal decomposition of the 2-step nilpotent Lie algebra $\nn$, such that  $\widetilde{\zz}$ is a subspace of the center $\zz$ and  the map $\pi:\zz \to \widetilde{\zz}$ denotes the orthogonal projection. 

Take the direct sum vector space $\widetilde{\nn}=\vv \oplus \widetilde{\zz}$ and define the Lie bracket by $$[V+Z,V'+Z']_{\widetilde{\nn}}=\pi([V,V']_{\nn}) \quad \mbox{ for all } V, W\in \vv, Z, Z'\in \widetilde{\zz}. $$
Equip the Lie algebra $\widetilde{\nn}$ with the restricted metric induced from $\la\,,\,\ra$ on $\nn$. 
Now, notice that  whenever the 2-step nilpotent Lie algebra  $\nn$ is  of type H (non-singular) then the Lie algebra  $\widetilde{\nn}$ is  a Lie algebra of Type H (non-singular). It is sufficient to observe that the maps $j_Z:\vv\to\vv$ are the same in both Lie algebras for every $Z\in\widetilde{\zz} \subset\zz$.

\section{Left-invariant 2-forms  and magnetic fields}\label{closedforms}

In this section we study closed invariant 2-forms, which are  called {\em magnetic fields},  on any 2-step nilpotent Lie group. We shall see that the closeness condition imposes some  restrictions. 

  Let $\omega$ denote a left-invariant 2-form on a  Lie group $(N, \la\,,\,\ra)$ equipped with a left-invariant metric. The metric on $N$ determines a metric on the corresponding Lie algebra $\nn$ such that  there is a unique skew-symmetric endomorphism $F:\nn \to \nn$ satisfying
$$\omega(X,Y)= \la F(X), Y\ra, \quad \mbox{ for all } X,Y\in \nn.$$

Conversely, let $F:\nn \to \nn$ denote a skew-symmetric endomorphism on the Lie algebra.  Define  $\omega$ as the associated 2-form given by $\omega(X,Y)= \la F(X),Y\ra$.

Thus, left-invariant 2-forms on $(N, \la\,,\,\ra$) are in correspondence with skew-symmetric endomorphisms $F\in\sso(\nn)$. Such skew-symmetric map $F$ is known as {\em Lorentz force}.

Now,  the condition of asking the 2-form $\omega$ to be closed  is equivalent to ask the skew-symmetric map $F$ to satisfy the equation:
\begin{equation}\label{clos}	\la F(U), [V,W]\ra + \la F(V), [W,U]\ra + \la F(W), [U,V]\ra=0, \qquad \mbox{ for all } U,V,W\in \nn.
	\end{equation}

Assume now  that the Lie group $N$ is  2-step nilpotent with Lie algebra $\nn$ which decomposes into the orthogonal splitting $\nn=\vv \oplus \zz$ as in Equation  \eqref{decomp2}. 
Write $F_{\vv}$ and $F_{\zz}$ for the corresponding projections onto the subspaces $\vv$ and $\zz$, respectively.  Then, the 2-form $\omega$ associated to $F$ is closed if and only if the following conditions hold
$$(\rm{C1})\qquad \qquad \qquad 	F_{\zz}(Z)\in C(\nn)^{\perp}, \qquad \mbox{ for all } Z\in \zz,  
$$
and
$$(\rm{C2}) \qquad 	\la F_{\zz}(U), [V,W]\ra + \la F_{\zz}(V), [W,U]\ra + \la F_{\zz}(W), [U,V]\ra=0, \quad \mbox{ for all } U,V,W\in \vv,
$$
which can be obtained by analyzing Condition \eqref{clos} in terms of the projections onto the subspaces $\vv$ and $\zz$.

\begin{rem}\label{ztrivial}
	Any skew-symmetric map on the Lie algebra $\nn$ such that the projection $F_{\zz}\equiv 0$ trivially satisfies Equations (C1) and (C2), since $\la F_{\zz}(U), [V,W]\ra\equiv 0$. In particular any exact 2-form satisfies this. 
\end{rem}

The skew-symmetric map $F:\nn \to \nn$ uniquely decomposes as
$$F=F_1+F_2,$$
where $F_1$ and $F_2$ are skew-symmetric maps such that, with respect to the orthogonal splitting  $\nn=\vv\oplus\zz$ in Equation \eqref{decomp2}, one has:
\begin{itemize}
	\item $F_1$ preserves the decomposition: $F_1(V+Z)=F_\vv(V) + F_\zz(Z)$, 
	\item $F_2$ interchanges the subspaces $\vv$ and $\zz$: $F_2(V+Z)=F_\zz(V) + F_\vv(Z)$ for all $V\in\vv$, $Z\in\zz$.
\end{itemize} 

In fact, for any skew-symmetric map on the Lie algebra $\nn$, $F\in \sso(\nn)$, one has the decomposition $F=F_{\zz}+F_{\vv}$ an so,  
$$F(V+Z)= (F_{\vv}(V)+F_{\zz}(Z))+( F_{\vv}(Z)+ F_{\zz}(V)), \quad \mbox{ for all } V\in \vv, Z\in \zz. $$
Thus, take the maps $F_1$ and $F_2$ respectively as above. 


The skew-symmetry property from $F$ implies that  both $F_1$ and $F_2$ are skew-symmetric. 

Notice that
\begin{itemize}
	\item $F_1$ trivially satisfies Condition (C2) and 
	\item $F_2$ trivially satisfies Condition (C1).
\end{itemize}

Thus the skew-symmetric map $F$ gives rise to a closed 2-form if and only if for both $F_1$ and $F_2$ it holds:
\begin{itemize}
	\item $F_1$  satisfies Condition (C1), that is $F_1(\zz)\subseteq \ker(j)$ and 
	\item $F_2$  satisfies Condition  (C2).
\end{itemize}

\begin{defn}\label{def1}
	Let $F$ denote a skew-symmetric map on a 2-step nilpotent Lie algebra $(\nn, \la\,,\,\ra)$. We say that
	\begin{enumerate}
		\item $F$ is of type I, if $F$ preserves the decomposition $\vv \oplus \zz$ (so, $F=F_1$),
\item $F$ is of type II, if $F(\vv)\subseteq \zz$ and $F(\zz)\subseteq \vv$ (so, $F=F_2$ above).
	\end{enumerate}
\end{defn}

Now, we shall study exact 2-forms. Start with a left-invariant 1-form $\eta$. Consider the linear isomorphism between the Lie algebra $\nn$ and its dual space $\nn^*$ given by the metric, that is, sending $U \to \ell_U$, where $\ell_{U}(V)=\la U, V\ra$. 

By considering the decomposition of the Lie algebra  $\nn$ given in  Equation \eqref{decomp2} one can write any left-invariant $1$-form $\eta$ as $\eta=\ell_{\widetilde{Z} +\widetilde{V}}$. Easily one verifies that the differential follows
$$d\ell_{\widetilde{Z} +\widetilde{V}}(U,V)=\la \widetilde{Z},[U,V]\ra=\la j(\widetilde{Z})U,V\ra.$$
Thus, the kernel of the differential operator $d:\{\mbox{1-forms on }\nn \} \to \{\mbox{ 2-forms on } \nn \}$ contains  the subspace $\{\ell_{{V}}, \,\mbox{ with } V\in \vv\}$, while for the description of  the rank of $d$ notice that  $d\eta\neq 0$ if and only if $j_{\widetilde{Z}}\neq 0$. That is, $\widetilde{Z}$ belongs to the commutator of $\nn$, $C(\nn)$. In this case,   for the 2-form $d\ell_{\widetilde{Z}}$, the corresponding Lorentz force $F$ is given by $F=j_{\widetilde{Z}}$ for some non-trivial $\widetilde{Z}\in C(\nn)$. 

We already proved the next result. 

\begin{prop} \label{prop1} Let $(N, \la\,,\,\ra)$ denote a 2-step nilpotent Lie group equipped with a left-invariant metric and Lie algebra $\nn$ with ortogonal splitting $\nn=\vv\oplus \zz$ as in \eqref{decomp2}. Let $F:\nn \to \nn$ denote a linear map. Write the map  $F$  as 
	$$F=F_1+F_2, $$
	where $F_1(\vv)\subseteq \vv$ and $F_1(\zz)\subseteq \zz$,  while $F_2(\zz)\subseteq \vv$ and $F_2(\vv)\subseteq \zz$. Then
	\begin{enumerate}[(i)]	
		\item The skew-symmetric map $F$ gives rise to a closed 2-form if and only if both maps $F_1$ and $F_2$ are skew-symmetric and
			\begin{itemize}	\item $F_1$  satisfies \rm{Condition } \rm{(C1)} and 
		\item $F_2$  satisfies \rm{Condition} \rm{(C2)}.

		\end{itemize}
		\item Exact left-invariant 2-forms are in one-to-one correspondence with skew-symmetric maps $j_{\widetilde{Z}}$ with $\widetilde{Z}\in C(\nn)$. 
	\end{enumerate}
\end{prop}
\begin{rem} Note that if the 2-step nilpotent Lie group $N$ has no Euclidean factor, then any skew-symmetric map $F$ of type I trivially satisfies Condition (C1), so that it always gives rise to a closed 2-form. Moreover, the linear map $F$ is trivial on the center.  
\end{rem}

\begin{example} Let $V_0+Z_0$ be any element on a 2-step nilpotent Lie algebra $\nn$. Then a natural choice for $F$ is the skew-symmetric part of $\ad(V)$: $\ad(V_0+Z_0)-\ad(V_0+Z_0)^*=\ad(V_0)-\ad(V_0)^*$, where $\ad(X)^*$ denotes the adjoint of  $\ad(X)$ with respect to the metric. Notice that $\ad(V_0)(\vv)\subseteq \zz$ and $\ad(V_0)^*(\zz)\subset \vv$. Thus $\ad(V_0)-\ad(V_0)^*$ gives rise to a closed 2-form if and only if Condition (C2) holds, equivalently:
	$$\la [V_0, V], [U,W]\ra + \la [V_0,U], [W,V]\ra + \la [V_0,W], [V, U]\ra=0\mbox{ for all } U,V,W\in \vv.$$
Which kind of elements $V_0\in \vv$ may satisfy this equation?
Take the Heisenberg Lie algebra of dimension 2n+1, $n\geq 2$, in Example \ref{ExHeis}. Assume $V_0=\sum_i x_iX_i + \sum_i y_iY_i$. By taking $U=Y_i, V=X_j$ and $W=Y_j$ with $i\neq j$, one gets $x_i=0$ for every $i$. Analogously, by taking $U=X_i$, one finally obtains $V_0=0$.  This is an example of a general statement proved in  Lemma \ref{closedF1}.
	\end{example}
\begin{example}\label{closedonh}
	Let $\hh_3$ denote the Heisenberg Lie algebra of dimension three. 	Let $e^i$, i=1,2,3 denote the dual basis of the orthonormal basis $e_1, e_2, e_3$. With the convention $e^{ij}=e^i\wedge e^j$, clearly the 2-forms $e^{12}, e^{13}, e^{23}$ give a basis of the space of 2-forms on $\nn$.  Since $\la j(e_3)e_1, e_2\ra=\la e_1, [e_2,e_3]\ra=1$, then $e^{12}$ is exact. 
	
	Moreover, 
	\begin{itemize}
		\item $e^{13}$ has associated the skew-symmetric map given by $F_{13}(e_1)=e_3$, $F_{13}(e_3)=-e_1$ and $F_{13}(e_2)=0$, 
		\item $e^{23}$ has associated the skew-symmetric map given by $F_{23}(e_2)=e_3$, $F_{23}(e_3)=-e_2$ and $F_{23}(e_1)=0$, 
			\end{itemize}
	which gives $F=\alpha j_{e_3}+ \beta F_{13} + \rho F_{23}$ with matrix representation in the orthonormal basis $e_1, e_2, e_3$
	$$F=\left( \begin{matrix}
		0 & - \rho & - \beta\\
		\rho & 0 & -\alpha\\
		\beta & \alpha & 0
	\end{matrix}
\right).
$$ The corresponding 2-forms are closed. 
In fact, a skew-symmetric map $F_2$ of type II will be  of the form $F_2=\beta F_{13} + \alpha F_{23}$. And easy computations show that the skew-symmetric map $F_2$  always satisfies Condition (C2). 

We just proved that any left-invariant 2-form on the Heisenberg Lie group of dimension three, $H_3$, is closed.
\end{example}

Notice that any skew-symmetric map of type II, namely  $F_2$, gives rise to a 2-form $\omega$ satisfying the condition 
$$\omega(Z,\widetilde{Z})=0, \quad \mbox{ for all } Z, \widetilde{Z}\in \zz.$$
This means, that the center is ``isotropic'' for $\omega$. In this situation, only one of  the two following conditions is true:
\begin{enumerate}[(i)]
		\item either $\omega(\zz, \nn)=0$ or
	\item there is $U\in \nn-\zz$ such that $\omega(Z, U)\neq 0$ for some $Z\in \zz$. 
\end{enumerate}
Indeed the non-trivial cases correspond to the second condition (ii). Due to the correspondence between 2-forms and skew-symmetric maps, one can obtain a description of 2-forms in terms of skew-symmetric maps of type I or II, as follows. 

As above, let $\la\,,\,\ra$ denote a metric on the 2-step nilpotent Lie algebra $\nn$ and take the orthogonal decomposition $\nn=\vv \oplus \zz$ as in Equation \eqref{decomp2}. Let $F:\nn \to \nn$ be the skew-symmetric map on $\nn$ such that $\omega(V+Z, \widetilde{V}+\widetilde{Z})=\la F(V+Z),\widetilde{V}+\widetilde{Z}\ra$. 

The condition of the center to be isotropic says that $\la F(Z),\widetilde{Z}\ra=0$ for all $Z, \widetilde{Z}\in \zz$, that is $F_{\zz}(Z)=0$, for any $Z\in \zz$,  which in terms of the families we introduce previously, gives:
$$F(Z+V)=F_1(V)+F_2(Z)+F_2(V)\quad \mbox{ for all } V+Z\in \nn.$$
But in this situation, the corresponding 2-form $\omega$ is closed if and only if
 $F_2$ satisfies Condition (C2). In fact, any skew-symmetric linear map $F_1$ with $F_1(\zz)\equiv 0$ is closed as already noticed  in Remark \ref{ztrivial}. 

On the other hand, $\omega(Z,U)\neq 0$ if and only if $\la F(Z), U\ra \neq 0$ which is equivalent to  $\la F_2(Z),U\ra \neq 0$ for some $Z\in \zz$ and some $U\in \nn-\zz$.  And this must occur for any metric. 

\begin{prop}\label{prop2} Let $\nn$ denote a 2-step nilpotent Lie algebra. \begin{enumerate}[(i)]
			\item There is a non-trivial closed 2-form $\omega$ for which either  $\omega(\zz, \zz)\neq 0$ or $\omega(\zz, \nn)=0$ if and only if  for any metric $\la\,,\, \ra$ on the Lie algebra $\nn$ there is a non-trivial skew-symmetric map  of type {\rm I} satisfying Condition {\rm(C1)}.
		\item There is a non-trivial closed 2-form $\omega$ for which $\omega(\zz, \zz)=0$ but $\omega(\zz, \nn)\neq 0$ if and only if for any metric $\la\,,\, \ra$ on the Lie algebra $\nn$ there is a non-trivial skew-symmetric map  of type {\rm II}  satisfying Condition {\rm (C2)}.
	
	\end{enumerate}
	\end{prop}
 
 
 Moreover, in terms of Definition \ref{def1}, the proposition above enables to distinguish 2-forms into two families. We shall say that a  2-form $\omega$ on the Lie algebra $\nn$ is
 \begin{enumerate}[(i)]
 	\item  {\em of type} { \rm I}: if it satisfies  either  $\omega(\zz, \zz)\neq 0$ or $\omega(\zz, \nn)=0$,
 	\item  {\em of type} {\rm{II}}: if it satisfies $\omega(\zz, \zz)=0$ but $\omega(\zz, \nn)\neq 0$.
 \end{enumerate}

 \begin{example} \label{hcomplex}
 	Let $\hh(\CC)$ denote the Heisenberg Lie algebra over $\CC$. Consider the underlying real Lie algebra of dimension six, that we denote in the same way. It has a center of dimension two spanned by $Z_1,Z_2$ and complementary subspace of dimension four spanned by the vectors $X_1,Y_1,X_2, Y_2$. They satisfy the non-trivial Lie bracket relations
 	$$[X_1,Y_1]=-[X_2,Y_2]=Z_1,\qquad [X_1,Y_2]=[X_2,Y_1]=Z_2.$$
 	The left multiplication by $i$ induces a real linear map $J:\hh(\CC) \to \hh(\CC)$ satisfying $J^2=-Id$ and $J\circ \ad(U)=\ad(U)\circ J$ for all $U\in \hh(\CC)$. Explicitly, in the basis, one has
 	$$J(Z_1)=Z_2, \quad J(X_1)=X_2, \quad J(Y_1)=Y_2.$$
 	Take the metric on $\hh(\CC)$ making  the set $Z_i,X_i,Y_i$ for  i=1,2, an orthonormal basis. Clearly the complex structure $J$ is skew-symmetric with respect to this metric $\la\,,\,\ra$. 
 	
 	Assume that $F$ is a skew-symmetric map on the Lie algebra giving rise to a closed 2-form. Write $F=F_1+F_2$ as above. Since this Lie algebra is non-singular, the kernel of $j$  is trivial, $\ker j=\{0\}$, and  the linear map  $F_1$ trivially satisfies Condition (C1). On the other hand, for   $V, W\in \vv$, the  condition (C2) for $F_2$ gives
 $$\begin{array}{rcl}
 	0 & = &\la F_2(V), [JV,W]\ra + \la F_2(JV), [W,V]\ra + \la F_2W, [V,JV]\ra\\
 &	= &\la F_2(V), J[V,W]\ra + \la F_2(JV), [W,V]\ra + \la F_2W, J[V,V]\ra\\
 &	= &\la F_2(V), J[V,W]\ra - \la F_2(JV), [V,W]\ra\\
 &	= & \la -JF_2(V)-F_2(JV), [V,W]\ra.
 \end{array}$$
 Since $W$ is an  arbitrary element and $\ad(V)$ is onto the center $\zz$ we get $	F_2(JV)=-JF_2(V)$,  for every $ V\in \vv$. And since $F_2$ is skew-symmetric it holds on $\nn$: 
 \begin{equation}\label{anticomp}
 	F_2 \circ J=-J\circ F_2.  
 \end{equation}
 Conversely,  any skew-symmetric map $F_2:\hh(\CC)\to \hh(\CC)$ that verifies Equation  \eqref{anticomp} gives rise to a closed 2-form of type II.
 \end{example}

 The next result determines a condition on the dimension of the Lie algebra and its center for the non-existence of  magnetic fields of type II, i.e. the non-existence of skew-symmetric maps of type II. See Proposition \ref{prop2}. 
 
 \begin{lem}\label{closedF1} 	Let $\nn$ denote a non-singular 2-step nilpotent Lie algebra such that $\dim \nn > 3 \dim \zz$. Then any closed 2-form on $\nn$ satisfies 
 	$$\omega(Z,U)=0, \quad \mbox{ for all } Z\in \zz, U\in \nn.$$
 \end{lem}
 \begin{proof}
  	Firstly, notice that since the Lie algebra is non-singular, its commutator coincides with the center,  $C(\nn)=\zz$. Let $\omega$ denote a closed two-form on $\nn$, then the non-singularity  property implies that  $\omega(\zz,\zz)=0$. In fact, let $Z,  \widetilde{Z}\in \zz$ with $Z=[U,V]$ for $U,V\in \nn$. The closeness condition says that $\omega(\widetilde{Z}, [U,V])=0$. 
 	
 	By the contrary, assume that there exists a non-trivial closed 2-form $\omega$ on the non-singular 2-step nilpotent Lie algebra $\nn$   such that there are  $Z\in \zz$ and  $U\in \nn-\zz$ satisfying  $\omega(Z,U)\neq 0$, that is $\omega$ is of type II. 	
 	
 	Let $\la\,,\,\ra$ be a metric on $\nn$ inducing a orthogonal decomposition   $\nn=\vv\oplus \zz$ as in Equation \eqref{decomp2}. Notice that by hypothesis, $\dim\vv> 2\dim \zz$. 
 	
 	Let $F$ denote the skew-symmetric map on the Lie algebra $\nn$ such that $\omega(X,Y)=\la F(X), Y\ra$. Then 
 	$$\omega(Z, U)\neq 0 \mbox{ if and only if } \la F(Z),U \ra \neq 0 \mbox{ if and only if }\la F_2(Z),U\ra\neq 0,$$
 	where $F_2$ is taken as in Proposition \ref{prop1}. Since $\la Z, F_2(U)\ra \neq 0$ says that the image of $F_2|_{\vv}$ is non-trivial,  we may assume (changing $Z$ if necessary) that $F_2(U)=Z$, so that $\la Z, F_2(U)\ra = \la Z, Z\ra\neq 0$. 
 	
 	Let $\mathcal W$ denote the kernel of $F_2$ on $\vv$, $\mathcal W=\ker(F_2|_{\vv})$. Thus one has:
 	$$\dim \mathcal W=\dim \vv - \dim Image(F_2|_{\vv})\geq \dim \vv - \dim \zz > \dim \vv/2$$
 	
 	Since $j_Z$ is non-singular, it also holds $\dim j_Z\mathcal W>\dim \vv /2$, which implies that the intersection is nontrivial,  $\mathcal W \cap j_Z \mathcal W\neq \{0\}$. Let $W, \widetilde{W}\in \mathcal W$ such that $j_Z W=\widetilde{W}\neq 0$. Now, the closeness condition for the 2-form $\omega$ is equivalent to Condition (C2) for $F_2$. And for $W, \widetilde{W}, U$ we get
 	$$0=\la F_2 W, [\widetilde{W}, U]\ra+\la F_2 \widetilde{W},[U,W]\ra+\la F_2 U, [W,\widetilde{W}]\ra= \la Z, [W, \widetilde{W}]\ra = \la \widetilde{W}, \widetilde{W}\ra\neq 0,$$
 	which is a contradiction. Thus, there are no closed 2-forms of type II under the hypothesis. 
 \end{proof}

  In the next section, we shall determine the Lie algebras of type H admitting Lorentz forces of type II. The previous result does not hold for singular or almost non-singular Lie algebras, proved below.

  \begin{example} {\bf A singular example.} Let $\nn$ denote the  Lie algebra $\nn=\vv \oplus \zz$, where $\vv$ is spanned by the vectors $V_1, V_2, V_3, V_4, V_5$ and $\zz$ is spanned by $Z_1, Z_2$, and they obey the non-trivial Lie bracket relations:
  $$[V_1, V_2]=Z_1, \qquad [V_3,V_4]=Z_2=[V_4,V_5].$$
  Take the metric on $\nn$ that makes of the previous basis an orthonormal basis. 
  
  Let $F: \nn \to \nn$ denote the skew-symmetric map given by
  $$F(V_3)=Z_2 =-F(V_5), \qquad F(V_4)=Z_1, \quad F(V_i)=0, i=1,2.$$
  Usual computations show that the 2-form given as $\omega(X,Y)=\la F(X),Y\ra$ is closed. 
  
  This is an example of a singular Lie algebra (that is, every $j_Z$ is singular) admitting a closed 2-form of type II. In fact, every map $j_{z_1Z_1 + z_2 Z_2}: \vv \to \vv$ have a matrix of the form
  $$\left( \begin{matrix}
  	0 & -z_1 & 0 & 0 & 0 \\
  	z_1 & 0  &  0 & 0 & 0\\
  	0 & 0 & 0 & -z_2 & 0 \\
  	0 & 0 & z_2 & 0 &-z_2 \\
  	0 & 0 & 0  & z_2 & 0
  \end{matrix}
  \right)
  $$
  in the basis of $\vv$ given above. And $\dim \nn> 3 \dim \zz$.

  \end{example}
  
   \begin{example} {\bf An almost non-singular example.} Let $\nn$ denote the  Lie algebra $\nn=\vv \oplus \zz$, where $\vv$ is spanned by the vectors $V_1, V_2, V_3, V_4$ and $\zz$ is spanned by $Z_1, Z_2, Z_3$, and they obey the non-trivial Lie bracket relations:
  	$$[V_1, V_2]=Z_1, \qquad [V_2,V_3]=Z_2, \qquad [V_3,V_4]=Z_3.$$
  	Take the metric on $\nn$ that makes the previous basis an orthonormal basis. 
  	
  	Any closed 2-form of type II on $\nn$ is associated to a skew-symmetric matrix  $F: \vv \to \zz$ of the form 
  	$$\left( 
  	\begin{matrix}
  		a & b & c & 0\\
  		-c & d & e & f\\
  		0 & -f & g & h
  		\end{matrix}
  	\right)
  	$$
  	with $a,b,c,d,e,f,g,h\in \RR$.  This follows from the closeness condition. On the other hand, any map $j_{z_1 Z_1+z_2Z_2+z_3 Z_3}$ is non-singular if $z_1\neq 0$ and $z_3\neq 0$, which says that $\nn$ is almost non-singular. In fact in the basis above, the matrix of $j_{z_1 Z_1+z_2Z_2+z_3 Z_3}$ is
  		$$\left( 
  	\begin{matrix}
  		0 & -z_1 & 0 & 0\\
  		z_1 & 0 & -z_2 & 0\\
  		0 & -z_2 & 0 & -z_3\\
  		0 & 0 & z_3 & 0 
  	\end{matrix}
  	\right).
  	$$ And  in this case, also $\dim \nn> 3 \dim \zz$. 
  	
  	Note that this Lie algebra is associated to a graph (see \cite{DM}).

  \end{example}

 Recall that for a 2-form $\omega$ its {\em kernel} is given by the following set:
 $$\ker(\omega)=\{ W\in \nn \, :\, \omega(W,V)=0 \mbox{ for all } V\in \nn\}.$$
 \begin{lem}\label{ImF} Let $\omega$ be a closed 2-form on the 2-step nilpotent Lie algebra $\nn=\vv \oplus \zz$ such that $\ker(\omega)\cap\zz=\{0\}$. Then the kernel of $\omega$,  $\ker(\omega)$, is an abelian subalgebra of $\nn$.	
 	
 	As a consequence, if the skew-symmetric map $F:\nn\to\nn$ is  a Lorentz force of type {\rm II} such that the image of the restriction satisfies  $Image\,F|_{\vv}=\zz$, then the kernel of $F$ in $\vv$, $\ker F|_{\vv}$,   is an abelian subalgebra of $\nn$. 
 \end{lem}
 
 \begin{proof}
 	Let $U,\,V\in\ker(\omega)$ and  let $W\in\nn$ be any element. By the closeness condition one has, 
 	\begin{align*}
 		0=\omega(U, [V,W])+ \omega(V, [W,U]) + \omega(W, [U,V])=\omega(W, [U,V]).
 	\end{align*}
 Since $W$ is arbitrary, it must be $[U,V] \in \ker(\omega)$. This implies  $[U,V]\in \ker(\omega)\cap\zz=\{0\}$, that is $[U,V]=0$ for all $U,V\in \ker(\omega)$ finishing the proof. 
 \end{proof}	
 
 \subsection{Uniform magnetic fields} 
 In the following paragraphs we study the existence of uniform magnetic fields, by making use of the formulas for the Levi-Civita connection.
 
 Let $\nabla$ denote the Levi-Civita connection on the 2-step Lie group  $(N, \la\,,\,\ra)$. Since the metric is invariant by left-translations, for $X, Y$ left-invariant vector fields one has the following formula for the covariant derivative:
 $$\nabla_X Y = \frac12 \{[X,Y]- \ad(X)^*(Y)-\ad(Y)^*(X)\}$$
 where $\ad(X)^*, \ad(Y)^* $ denote the adjoints of $\ad(X), \ad(Y)$ respectively.  Thus, one obtains
 $$\left\{ 
 \begin{array}{lll}
 	(i) & \nabla_Z \widetilde{Z}=0 & \mbox{ for all } Z, \widetilde{Z}\in \zz,\\
 	(ii) & \nabla_Z X= \nabla_X Z = -\frac12 j_Z X & \mbox{ for all } Z\in \zz, X\in \vv,\\
 	(iii) & \nabla_X \widetilde{X}=\frac12 [X,\widetilde{X}] & \mbox{ for all } X, \widetilde{X}\in \vv.
 \end{array}
 \right.
 $$

 \begin{defn}
 	Let $F$ denote a skew-symmetric map on a 2-step nilpotent Lie algebra $(\nn, \la\,,\ra)$. The magnetic field is called {\em uniform} whenever it is parallel, equivalently, the corresponding skew-symmetric map  is parallel, i.e.  $\nabla F\equiv 0$.
 \end{defn}

\smallskip

$\bullet$ {\bf Parallel Lorentz forces of type I.} Take $F_1$ be a Lorentz force of type I. Recall that the closeness condition says that $F_1(Z)\in\ker(j)$ for all $Z\in \zz$.

\begin{enumerate}[(i)]
	\item  For $Z, \widetilde{Z}\in \zz$ one has $\nabla_Z F_1(\widetilde{Z})=0=F_1 \nabla_Z \widetilde{Z}$. 
	\item For $Z\in \zz$, $V\in \vv$:
\begin{itemize}
	\item 	$\nabla_Z F_1(V)=-\frac12 j_Z(F_1(V))$ and $F_1 \nabla_Z V=-\frac12 F_1(j_ZV)$, which implies
	$$F_1 \circ j_Z =j_Z \circ F_1, \quad \mbox{ for } Z\in \zz.$$
	\item $\nabla_V F_1(Z)=-\frac12 j_{F_1(Z)}(V)$ and $F_1 \nabla_V Z=-\frac12 F_1(j_Z V)$, which implies
	$$F_1 \circ j_Z =j_{F_1(Z)}, \quad \mbox{ for } Z\in \zz.$$
\end{itemize}
\item For $V, \widetilde{V}\in \vv$: the condition
$\nabla_V F_1(\widetilde{V})= F_1 \nabla _V \widetilde{V}$ implies $$F_1\circ \ad(V)=\ad(V) \circ F_1, \quad \mbox{ for all } V\in \vv.$$
\end{enumerate}

\begin{cor}
	Let $(N,\la\,,\,\ra)$ denote a 2-step nilpotent Lie group equipped with a left-invariant metric. If $N$ is either  non-singular or it has no Euclidean factor and it is  almost non-singular,  then any uniform left-invariant magnetic field of type I on $N$ is trivial.
\end{cor}
Observe that  since $\ker(j)$ is trivial, one has $F_1(\zz)\equiv 0$. From the conditions above it follows $0=j_{F_1(Z)}=j_Z \circ F_1$ for all $Z\in \zz$.  By taking  $Z\in \zz$ such that $j_Z:\vv\to \vv $ is non-singular, it follows that $F_1(V)=0$, for every $V\in \vv$.

\smallskip

$\bullet$ {\bf Parallel Lorentz forces of type II.} Let $F_2$ be a Lorentz force of type II. Recall that the closeness condition says that $F_2$ must satisfy Condition (C2).

 \begin{enumerate}[(i)]
 	\item  For $Z, \widetilde{Z}\in \zz$ one has $\nabla_Z F_2(\widetilde{Z})=-\frac12 j_Z(F_2(\widetilde{Z}))$. Also $F_2 \nabla _{Z}\widetilde{Z}=0$. Therefore the parallelism implies  $ j_Z(F_2(\widetilde{Z}))=0$. 
 	\item For $Z\in \zz$, $V\in \vv$:
 	\begin{itemize}
 		\item 	$\nabla_Z F_2(V)= 0$ and $F_2 \nabla_Z V=-\frac12 F_2(j_ZV)$, which implies for $F_2$ parallel
 		$$F_2 \circ j_Z=0, \quad \mbox{ for } Z\in \zz.$$
 		\item $\nabla_V F_2(Z)=\frac12 [V, F_2(Z)]$ and $F_2 \nabla_V Z=-\frac12 F_2 \circ j_ZV$, that implies for $F_2$ parallel
 		$$F_2 \circ j_Z =  \ad(F_2(Z)), \quad \mbox{ for } Z\in \zz, $$
 		which is equivalent to $j_{\widetilde{Z}}(F_2(Z))=j_Z (F_2\widetilde{Z})$, for all $Z, \widetilde{Z}\in\zz$. 
 	\end{itemize}
 	\item For $V, \widetilde{V}\in \vv$: the condition
 	$\nabla_V F_2(\widetilde{V})= F_2 \nabla _V \widetilde{V}$ implies $$j_{F_2(V)}= F_2 \circ \ad(V), \quad \mbox{ for all } V\in \vv.$$
 \end{enumerate}
 
 \begin{cor}
 	Let $(N,\la\,,\,\ra)$ denote a 2-step nilpotent Lie group equipped with a left-invariant metric. If $N$ is either  non-singular or it is almost non-singular,  then any uniform left-invariant magnetic field of type II on $N$ is trivial.
 \end{cor}

The proof follows from $F_2 \circ j_Z V =0$ and $ j_Z F_2(\widetilde{Z})=0$  for all $Z, \widetilde{Z}\in \zz$ and $V\in \vv$, and observing that by the hypothesis, there always exists $Z\in \zz$ such that $j(Z)$ is non-singular. This gives $F_2\equiv 0$.

 \section{Lorentz forces on Lie groups of type H} In this section we focus on Lie groups of type H, whose Lie algebra is of type H. The goal is to determine exactly which ones admit non-trivial closed 2-forms of type II, equivalently skew-symmetric maps of type II satisfying Condition (C2). 
  
  Now, let $(\nn, \la\,,\,\ra)$ denote an $H$-type Lie algebra with $\dim\vv=n$ and $\dim\zz=m$. By Lemma \ref{closedF1}, the existence of a non-trivial skew-symmetric map $F:\nn\to\nn$ of type II satisfying Condition  (C2), says $n\leq 2m$. This and item $(5)$ in Proposition \ref{pairs}  implies that 
   $$n\leq 2\,m<2\rho(n).$$ Thus, the only possible pairs $(n,m)$ that verify this inequality are $(2,1)$, $(4,2)$, $(4,3)$, $(8,4)$, $(8,5)$, $(8,6)$, $(8,7)$ and $(16,8)$. These are the ones described on section \ref{general}.

   First we give a useful lemma concerning construction (iii) of a 2-step nilpotent Lie algebra $\widetilde{\nn}=\vv \oplus \widetilde{\zz}$ from a 2-step nilpotent Lie algebra $\nn=\vv \oplus \zz$, given at the end of section \ref{general}.


 
  \begin{lem}\label{ztilde} Let $\nn=\vv \oplus \zz$ denote a non-singular 2-step nilpotent Lie algebra. Let $\widetilde{\zz}\subseteq \zz$ be any non-trivial subspace. Let $\widetilde{\nn}=\vv \oplus \widetilde{\zz}$ denote the 2-step nilpotent Lie algebra equipped with the restricted Lie bracket and metric, given by construction (iii). 
  	Then any non-trivial Lorentz force on  $\widetilde{\nn}$ gives rise to a non-trivial Lorentz force on $\nn$.
  	
  	Reciprocally, any non-trivial Lorentz force on $\nn$ such that $F(\vv)\subset\widetilde{\zz}$ restricts to a non-trivial Lorentz force on  $\widetilde{\nn}$.
  \end{lem}
  \begin{proof}
  Suppose that $\widetilde{\nn}$ admits a non-trivial skew-symmetric linear map $\widetilde{F}:\widetilde{\nn}\to \widetilde{\nn}$ giving rise to a closed 2-form. Define a linear map on $\nn$,  ${F}:\nn\to\nn$ by the formula
 $${F}(V+Z)=\widetilde{F}(V+\pi(Z))\quad \mbox{ for }V\in\vv, Z \in\zz.$$
 
 We shall prove that ${F}$ is a  skew-symmetric map  giving rise to a closed 2-form on $\nn$. Moreover, by writing $\widetilde{F}=\widetilde{F}_1+\widetilde{F}_2$ and $F=F_1+F_2$ as in Proposition \ref{prop1}, then ${F}_2 \neq 0$ if $\widetilde{F}_2\neq 0$. 
 
 Now, 
 \begin{itemize}
 	\item The map ${F}$ is skew-symmetric. In fact, take $V,V'\in\vv$, $Z,Z'\in\zz$, it holds
 	$$\begin{array}{rcl}
 		\la {F}(V+Z), V'+Z'\ra &  = & \la \widetilde{F}(V+\pi(Z)), V'+\pi(Z')\ra \\
 		&  = & -\la V+\pi(Z), \widetilde{F}(V'+\pi(Z'))\ra \\
 		&  = & -\la V+Z,F(V'+Z')\ra, 
 	\end{array}
 	$$
 	since $Image(\widetilde{F})\subset\vv\oplus\widetilde{\zz}$.
 	
 	\item The map ${F}$ trivially satisfies condition (C1) since $\nn$ is non-singular. It also satisfies Condition (C2). In  fact, let $U,\,V,\,W\in\vv$ then
 	\begin{align*}
 		\la {F}(U), [V,W]\ra + \la {F}(V), [W,U]\ra + \la {F}(W), [U,V]\ra \\
 		=\la \widetilde{F}(U), \pi([V,W])\ra + \la \widetilde{F}(V), \pi([W,U])\ra + \la \widetilde{F}(W), \pi([U,V])\ra\\
 		=\la \widetilde{F}(U), [V,W]_{\widetilde{\nn}}\ra + \la \widetilde{F}(V), [W,U]_{\widetilde{\nn}}\ra + \la \widetilde{F}(W), [U,V]_{\widetilde{\nn}}\ra =0
 	\end{align*}
 	where in the second equality we use again that $Image\,F\subset\vv\oplus\widetilde{\zz}$ and in the last equality the fact that $\widetilde{F}$ is closed in $\widetilde{\nn}$.
 	
 	\item Since ${F}(V)=\widetilde{F}(V)$ for every $V\in\vv$, one immediately gets that ${F}_2\neq 0$ if $\widetilde{F}_2\neq 0$. 
 \end{itemize}
 \end{proof}

 We conclude that if a Lie algebra of type H with pair  $(n,m)$  does not admit a  non-trivial Lorentz force $F$ of type {\rm II} (and satisfying Condition (C2)) then neither do the Lie algebras of type H of pair $(n,m')$ for $m'<m$, in the cases where the pair determines the Lie algebra.

 \smallskip
 
{\bf Lorentz forces for the pair (8,7).} Now we shall prove the non-existence of skew-symmetric maps $F$ satisfying condition (C2) on the   octonionic Heisenberg algebra $\hh(\OO)$. 
 
 Let $\hh(\OO)=\vv \oplus \zz$ be the corresponding orthogonal decomposition with associated pair $(8,7)$. Assume that the linear  map  $F:\vv\oplus\zz\to\vv\oplus\zz$ is a skew-symmetric map of type {\rm II} satisfying condition (C2). Since $F(\vv)\subset\zz$, and $\dim \zz < \dim \vv$, the kernel of $F:\vv \to \zz$ must be non-trivial, i.e. there exists $U\in\vv$ with $|U|=1$ such that $F(U)=0$. 
 
 Since the group of orthogonal automorphisms of $\vv\oplus\zz$ acts transitively on the unit sphere of $\vv$, see \cite{R}, one may assume that  $U=1\in\OO$. More precisely, consider the map $\phi^{-1} F \phi$ where $\phi$ is an orthogonal automorphism such that $\phi(1)=U$.
 
 Define a linear  map on the center $G:\zz\to\zz$ by $G(Z)=F(j(Z)U)$. Since the map $F$ satisfies condition (C2), for $U=1\in \OO$ and $Z,Z'\in \zz$ it holds
 $$\begin{array}{rcl}
 	0 & = & \la F(U), [j_ZU,j_{Z'}U]\ra + \la F(j_ZU), [j_{Z'}U,U]\ra + \la F(j_{Z'}U), [U,j_ZU]\ra \\
 	& 	= &  -\la G(Z), Z'\ra + \la G(Z'), Z\ra, 
 \end{array}
 $$
 which says that $G$ is a symmetric map.
 
 Let $Z_1$, $Z_2$, $Z_3$ be  orthonormal vectors on $\zz$. Again, since the skew-symmetric map $F$ satisfies Condition (C2), one has
 \begin{equation}\label{ec12}
 	\begin{array}{rcl}
 		0 & = & 	\la G(Z_1), [j_{Z_2}U,j_{Z_3}U]\ra + \la G(Z_2), [j_{Z_3}U,j_{Z_1}U]\ra + \la G(Z_3), [j_{Z_1}U,j_{Z_2}U]\ra \\
 		&	=  & \la G(Z_1), [Z_2,Z_3]\ra + \la G(Z_2), [Z_3,Z_1]\ra + \la G(Z_3), [Z_1,Z_2]\ra \\
 		&	=  & \la G(Z_1), Z_2.Z_3\ra + \la G(Z_2), Z_3.Z_1\ra + \la G(Z_3), Z_1.Z_2\ra.
 	\end{array}
 \end{equation}
 Take the orthonormal vectors in the center  $Z_2.Z_3$, $Z_3.Z_1$, $Z_3$. Thus,  one has
 \begin{equation}\label{ec13}
 	\begin{array}{rcl}
 		0 & = &	\la G(Z_2.Z_3), (Z_3.Z_1).Z_3\ra + \la G(Z_3.Z_1), Z_3.(Z_2.Z_3)\ra + \la G(Z_3), (Z_2.Z_3).(Z_3.Z_1)\ra\\
 		&	= & \la G(Z_2.Z_3), Z_1\ra + \la G(Z_3.Z_1), Z_2\ra + \la G(Z_3), (Z_2.Z_3).(Z_3.Z_1)\ra\\
 		&	= & \la G(Z_2.Z_3), Z_1\ra + \la G(Z_3.Z_1), Z_2\ra - \la G(Z_3), Z_3.(Z_1.Z_2).Z_3\ra,
 \end{array}\end{equation}
 where we are using that $Z_1$, $Z_2$ and $Z_3$ anti-commute and that the octonions is an alternative algebra \cite{KS}.
 
 By comparing Equations \eqref{ec12} and \eqref{ec13}, we conclude that $$\la G(Z_3), Z_1.Z_2\ra = -\la G(Z_3), Z_3.(Z_1.Z_2).Z_3\ra,$$ for any set of  orthonormal elements in the center,  $Z_1$, $Z_2$, $Z_3$ in $\zz$. 
 
  For any $W\perp Z_3$ with $|W|=1$, take $Z_1 \perp span\{W,\, Z_3,\, Z_3.W\}$ with norm $1$. Then $Z_1,\, Z_2=W.Z_1,\, Z_3$ are orthonormal with $Z_1.Z_2=W$ and
 $$\la G(Z_3), W\ra =-\la G(Z_3), Z_3.(W.Z_3)\ra=-\la G(Z_3), W\ra =0.$$
 Since $Z_3$ is arbitrary, there exist $a\in\RR$ such that $G(Z)=a Z$ for all $Z\in\zz$. Taking $Z_1=i$, $Z_2=j$ and $Z_3=k$, the imaginary unit quaternions inside $\OO$, in eq. (\ref{ec12}) we have
 $$0=\la G(i), i\ra + \la G(j), j\ra + \la G(k), k\ra =3a.$$
 Therefore the map $G$ is trivial and so, this implies that the linear map $F$ is  trivial, $F\equiv 0$.
 
 \smallskip

{\bf Lorentz forces for the pair (16,8).} Now, let $\nn=\vv \oplus \zz$ denote  the $H$-type Lie algebra associated to the pair $(16,8)$. Thus, according to the section \ref{extypeH}, we work with the octonions $\mathbb{A}=\OO$, with $\vv=\OO\times \OO$ and $\zz=\OO$.

Let $F$ denote a skew-symmetric map  of type II   satisfying Condition (C2) on the Lie algebra $\nn$. Then $Image\,F|_{\vv} = \zz$, otherwise it will induce a Lorentz force of type II on a $H$-type Lie algebra with pair $(16,m')$ for $m'< 8$, by lemma \ref{ztilde}, which is a contradiction. 

 By Lemma \ref{ImF},  the kernel of the map $F$,  $ker\,F|_{\vv}$ is an abelian subalgebra of dimension $8$ on the subspace $\vv$, and by Lemma \ref{maxabelianHO} we can assume  that $ker\,F|_{\vv}=\mathbb{A}\times\{0\}$. Define the linear map on the octonions $L:\mathbb{O} \to\mathbb{O}$ by the formula $L(Z)= F(0,Z)$. Take the elements $Z_1,\,Z_2, Z_3\in \mathbb{O}$, and compute
 \begin{align*}
 	0=\la F(Z_3,0), [(0,Z_1), (0,Z_2)]\ra + \la F(0,Z_1), [(0,Z_2), (Z_3,0)]\ra + \la F(0,Z_2), [(Z_3,0), (0,Z_1)]\ra\\
 	= -\la L(Z_1), [(Z_3,0), (0,Z_2)]\ra + \la L(Z_2), [(Z_3,0), (0,Z_1)]\ra\\
 	=-\la L(Z_1), Z_3.Z_2\ra + \la L(Z_2), Z_3.Z_1\ra\\
 	=-\la L(Z_1).\bar{Z_2}, Z_3\ra + \la L(Z_2).\bar{Z_1}, Z_3\ra\\
 	=\la -L(Z_1).\bar{Z_2} + L(Z_2).\bar{Z_1}, Z_3\ra.
 \end{align*}
 Since $Z_3$ is arbitrary, one gets that 
 \begin{equation}\label{eqH168}
 	-L(Z_1).\bar{Z_2} + L(Z_2).\bar{Z_1} = 0
 \end{equation}
 for every $Z_1,\,Z_2\in\zz$. Choose the element   $Z_1=1$, to have
 \begin{equation}\label{eqH1682}
 	L(Z) = L(1).\bar{Z} 
 \end{equation}
 for every $Z\in\zz$. By replacing in Equation \eqref{eqH168}, one obtains 
 \begin{equation*}\label{eqH1683}
 	-(L(1).\bar{Z_1}).\bar{Z_2} + (L(1).\bar{Z_2}).\bar{Z_1} = 0
 \end{equation*}
 for every $Z_1,\,Z_2\in\zz$. If we take $Z_1$ and $Z_2$ on a quaternionic subalgebra that contains $L(1)$, the associativity property holds and one gets 
 \begin{equation*}
 	L(1).(-\bar{Z_1}.\bar{Z_2} + \bar{Z_2}.\bar{Z_1}) = 0. 
 \end{equation*}
 But  this subalgebra is non-commutative. Therefore,  we get that $L(1)=0$ and, from \eqref{eqH1682}, it finally holds $L\equiv 0$. So, there is no skew-symmetric map of type II satisfying Condition (C2) in this case. 
 
 This proves the next result. 
  \begin{prop}\label{force-oct}  On the algebras of type H, $\hh(\OO)$ and $\nn=(\OO \times \OO)\oplus \OO$ as above, any Lorentz force is of type {\rm I}. 
  \end{prop}
  
  This finally enables to determine which Lie algebras of type H admit Lorentz forces of type II. 
  
 \begin{thm}\label{thm2}  Let $\nn=\vv\oplus\zz$ be a Lie algebra of type H. Then $\nn$ admits  a Lorentz force of type {\rm II} if and only if $\nn$ is the $3$-dimensional Heisenberg algebra, the $6$-dimensional complex Heisenberg algebra or the $7$-dimensional quaternionic Heisenberg algebra.
 \end{thm}
 
 \begin{proof}
 	 	As already said, the only algebras of type H that could admit a Lorentz force of type II will have possible pairs $(n,m)$ which are $(2,1)$, $(4,2)$, $(4,3)$, $(8,4)$, $(8,5)$, $(8,6)$, $(8,7)$ and $(16,8)$. Proposition \ref{force-oct} discards the cases for the pairs $(16,8)$ and $(8,7)$, and from the last proposition and by Lemma \eqref{ztilde} also discard the cases $(8,4)$, $(8,5)$, $(8,6)$. 
 	
 	Examples \eqref{closedonh} and \eqref{hcomplex} show the existence of Lorentz forces of type II on the Heisenberg Lie algebra of dimension three $\hh_3$ and the complex Heisenberg Lie algebra $\hh(\CC)$, respectively. 
 	
 	On the $7$-dimensional quaternionic Heisenberg algebra, in terms of the construction with the division algebra of quaternions, we may assume that the set  $\{1,i,j,k\}$ is a basis of $\vv=\mathbb{H}$. Thus,  Condition (C2) for a skew-symmetric map on this Lie algebra,  reduces to the next four linear equations:
 	\begin{align*}
 		\la F(1), k\ra - \la F(i), j\ra + \la F(j), i\ra = 0,\\
 		-\la F(1), j\ra - \la F(i), k\ra + \la F(k), i\ra = 0,\\
 		\la F(1), i\ra - \la F(j), k\ra + \la F(k), j\ra = 0,\\
 		\la F(i), i\ra + \la F(j), j\ra + \la F(k), k\ra = 0, 
 	\end{align*}
 	whose solution space is an $8$-dimensional subspace. This gives the set of all possible Lorentz forces of type II, $\{F:\vv\to\zz\}$ and each solution extends in the obvious way to a  skew-symmetric map of type II.
 	
 \end{proof}

\section{Other examples with trivial magnetic fields of type II}

In this section we show other non-singular Lie algebras admitting no closed 2-forms of type II. We obtain then by modifying $H$-type Lie algebra as follows. 

Let $\hh=\vv\oplus\zz$ be an $H$-type Lie algebra and choose $Z_1\in\zz$ with $||Z_1||=1$. Let $\vv=\vv_1\oplus\vv_2$ be an orthogonal decomposition with the condition that $\vv_1$ and $\vv_2$ are invariant under $j_{Z_1}$ and $j_{Z_2}$ for some $Z_2\perp Z_1$: this is possible for $dim\,\hh >7$. Observe that both $\vv_1$ and $\vv_2$ are modules of the Clifford algebra $\mathrm{Cl}(2)$, hence their dimensions must be of the form  $4s$ for some non-negative integer $s$. Take a scalar  $r\in\RR$. 

Define a new Lie algebra structure on the metric vector space $\nn_r:=\vv\oplus\zz$ by mean of the next skew-symmetric maps
$$\tilde{j}_{Z+\lambda Z_1}(V_1+V_2) = j_{Z+\lambda Z_1}(V_1+V_2) + (r-1) j_{\lambda Z_1} V_2$$    
where $Z\in\langle Z_1\rangle ^{\perp}$, $\lambda\in\RR$, $V_i\in\vv_i$ for $i=1,2$. In simple terms, we modify $j_{Z_1}$ by a factor $r$ on the subspace $\vv_2$. 

Notice that for a fix $r\neq 0$ there are several different new Lie algebras. In fact, the resulting Lie algebra depends on $Z_1$ and on the decomposition $\vv=\vv_1\oplus\vv_2$. However for $r=1$ one gets an $H$-type Lie algebra. 

\begin{thm}\label{other-nonsing} 
	For each $r>0$, every 2-step Lie algebra $\nn_r$ is non-singular and not isomorphic to an $H$-type Lie algebra if $r\neq 1$. 
\end{thm}

\begin{proof}
Let $0\neq Z+\lambda Z_1\in\zz$ with $Z\perp Z_1$ and $V_1 \in \vv_1$, $V_2\in\vv_2$ such that
\begin{align*}
\tilde{j}_{(Z+\lambda Z_1)}(V_1+V_2)=0 \implies j_{(Z+\lambda Z_1)}\tilde{j}_{(Z+\lambda Z_1)}(V_1+V_2)=0 \\
\implies -||Z+\lambda Z_1||^2 (V_1+V_2)+(r-1)j_{(Z+\lambda Z_1)}j_{\lambda Z_1} V_2=0\\
\implies -||Z+\lambda Z_1||^2 (V_1+V_2)+(r-1)j(Z)j_{\lambda Z_1}V_2  - (r-1)\lambda^2 V_2 =0.  
\end{align*}
Taking inner product with $V_2$:
\begin{align*}
-||Z+\lambda Z_1||^2 ||V_2||^2 -(r-1)\lambda^2 ||V_2||^2 = 0
\implies (-||Z||^2-r\lambda^2)||V_2||^2=0.
\end{align*}
Since $r>0$ this implies that $Z+\lambda Z_1=0$ or $V_2=0$, but if $V_2=0$ then replacing in the previous equation we obtain that  $V_1=0$.	
	
To prove that these Lie algebras are not of  type $H$ we will show that they are not of $\tilde{H}-$type as in \cite{LT}. This means that $Z\to \det(\tilde{j}_Z^2)^{1/n}$ is not a quadratic form, where $n =dim\,\vv$. 
	
Under our hypothesis there is $Z_2\in\zz$, $Z_2\perp Z_1$, such that $J_{Z_2}$ preserves the subspaces $\vv_1$ and $\vv_2$. We supposed without losing generality that $||Z_2||=1$. For $\mu,\lambda\in\RR$,
\begin{align*}
\tilde{j}_{\mu Z_2+\lambda Z_1}^2 (V_1+V_2)  = \tilde{j}_{\mu Z_2+\lambda Z_1} ( j_{\mu Z_2+\lambda Z_1}(V_1+V_2) + (r-1) j_{\lambda Z_1} V_2)\\
= j_{\mu Z_2+\lambda Z_1}(j_{\mu Z_2+\lambda Z_1}(V_1+V_2) + (r-1) j_{\lambda Z_1} V_2)+ (r-1)j_{\lambda Z_1}(j_{\mu Z_2} V_2 + r j_{\lambda Z_1} V_2) \\
=-||\mu Z_2+\lambda Z_1||^2 (V_1+V_2) - (r-1)\lambda^2 V_2 - (r-1)r \lambda^2 V_2\\
= -(\mu^2+ \lambda^2) V_1 - (\mu^2+ r^2\lambda^2)V_2. 
\end{align*}
Then
$$q(\mu Z_2+\lambda Z_1) = \det(\tilde{j}_{\mu Z_2+\lambda Z_1} ^2)^{1/n} = - (\mu^2 + \lambda^2)^{n_1/n}(\mu^2+ r^2\lambda^2)^{n_2/n}$$
where $n_1=dim\, \vv_1$ and $n_2=dim\, \vv_2$.  This form is quadratic if and only $r=1$. One way to check this is to see that
$$\frac{\partial q}{\partial\mu\partial \lambda}(\mu Z_2+\lambda Z_1) = -\frac{4 n_1 n_2 (r^2 -1)^2 \mu^3 \lambda^3}{(\mu^2+ \lambda^2)^{2-\frac{n_1}{n}} (\mu^2+ r^2\lambda^2)^{2-\frac{n_2}{n}}}$$  
is constant with respect to $\mu$ and $\lambda$ if and only if $r=1$, since $r>0$. 
\end{proof}

\begin{rem}
	If $r\leq 0$ the Lie algebras $\nn_r$ are almost non-singular. This is trivial if $r=0$ since $j_{Z_1}V_2=0$ for $V_2\in\vv_2$ and if $r<0$, taking $0\neq V_2\in\vv_2$ such that $V_1=\frac{r}{||Z||^2}j_Z j_{Z_1} V_2 \in\vv_1$ for some $0\neq Z\in\langle Z_1\rangle ^{\perp}$ and $\lambda =\tfrac{||Z||}{\sqrt{|r|}}$  then
	$$\tilde{j}_{Z+\lambda Z_1} (V_1+V_2) =0.$$
	
	For any $H$-type Lie algebra of dimension greater than $7$ a particular decomposition of $\vv$ is obtain in the following way. Suppose $dim\,\zz = m$. Let $\tilde{m}<m$ be the greatest number such that the dimension of a irreducible $\mathrm{Cl}(\tilde{m})$-module is half of the dimension of a irreducible $\mathrm{Cl}(m)$-module. Let $B=\{Z_1,\ldots,Z_{m}\}$ be an orthonormal basis of $\zz$, then the subset $\{Z_1,\ldots,Z_{\tilde{m}}\}\subset B$ generates the Clifford algebra  $\mathrm{Cl}(\tilde{m})$ whose action on $\vv$ decomposes as $\vv=\vv_1\oplus\vv_2$ where $dim\,\vv_1=dim\,\vv_2$. Then every linear map $j_{Z_i}$ for $1\leq i\leq \tilde{m}$ preserves the decomposition  $\vv_1\oplus\vv_2$  while: $j_{Z_i}\vv_1=\vv_2$ and $j_{Z_i}\vv_2=\vv_1$ for $\tilde{m}<i\leq m$. Notice that this procedure is valid for an irreducible or non-irreducible $H$-type Lie algebra, by decomposing each irreducible component. In this way we can consider the corresponding modified Lie algebras $\nn_r$ for this decomposition.
	
	 In this case, $n_r$ is isomorphic to $n_{1/r}$ for every $r>0$. The map $\phi:\nn_r\to\nn_{1/r}$ given by $\phi|_{\vv}=j_{Z_1}j_{Z_2}\ldots j_{Z_{\tilde{m}}}j_{Z_{\tilde{m}+1}}$, $\phi(Z_1)=\tfrac{(-1)^{\tilde{m}}}{r}Z_1$, $\phi(Z_i)=(-1)^{\tilde{m}}Z_i$ for $i=2,\ldots, \tilde{m}+1$ and $\phi(Z_i)=(-1)^{\tilde{m}+1}Z_i$ for $i=\tilde{m}+2,\ldots,m$; is an isomorphism.  This can be verified seeing that $$\phi^{t} j^r_{Z_i} \phi = j^{1/r}_{\phi^{t}(Z_i)}, \text{ for } i=1,\ldots,m,$$
	  where $j^r$ and $j^{1/r}$ are the $j$-operators of $n_r$ and $n_{1/r}$, respectively. For $i=1$ we use the fact that $\phi$ interchanges the spaces $\vv_1$ and $\vv_2$.  
	
\end{rem}

\smallskip

{\bf Case  (8,7)} Now we study the existence of magnetic fields of type II on the Lie algebras $\nn_r$ derived from  the octonionic $H$-type Lie algebra (8,7). 

Recall that the maps $j_z$ have matricial presentations as follows in the basis $V_1,V_2,\hdots, V_8$ and $Z=z_1 Z_1 + \hdots +z_7 Z_7$:
\begin{equation}\label{j87}
	j_Z= \left(\begin{matrix}
		0& -z_1 & -z_2  &  -z_3 & -z_4 & -z_5 & -z_6 & -z_7\\
		z_1& 0 &  -z_3 & z_2  & z_5 & -z_4 & z_7 & -z_6 \\
		z_2& z_3 &  0 & -z_1  & z_6 & -z_7 & -z_4 & z_5\\
		z_3& -z_2 & z_1 & 0  & z_7 & z_6 & -z_5 & -z_4\\
		z_4& -z_5 & -z_6  & -z_7  & 0 & z_1 & z_2 & z_3\\
		z_5&  z_4&  z_7 &  -z_6 & -z_1 & 0 & z_3 & -z_2\\
		z_6& -z_7 &  z_4 & z_5  & -z_2 & -z_3 & 0 & z_1\\
		z_7&  z_6& -z_5  & z_4  & -z_3 & z_2 & -z_1 & 0\\
	\end{matrix}
	\right).
\end{equation}

Using the previous remark, in this case $\tilde{m} = 3$ and we obtain the following subspaces $\vv_1=span\{V_1,V_2,V_3,V_4\}$ and $\vv_2=span\{V_5,V_6,V_7,V_8\}$, so that $\vv=\vv_1 \oplus \vv_2$ is invariant under $j_{Z_1}$ and $j_{Z_2}$. Define the Lie algebra $\nn_r$ as above.

To determine the existence of magnetic fields of type II on $\nn_r$ we write down the linear system of 56 equations and 56 variables derived from Condition (C2):
$$\la [V_i,V_j], F(V_k)\ra + \la [V_j,V_k], F(V_i)\ra + \la [V_k,V_i], F(V_j)\ra=0.$$

To solve this, one takes $F(V_s)=\sum_{t=1}^7 a_{ts} Z_s$. The solution of the corresponding linear system can be achieved  with help of any software.
 One gets the trivial solution,  $F=0$. 

\begin{rem}
	For the Lie algebra above $\nn_r$, the corresponding Pfaffian is given by:
	$$\mathrm{Pf}(J_Z)=\|Z\|^4+z_1^2((r^2-1)(z_1^2+z_2^2+z_3^2) + 2(r-1)(z_4^2+z_5^2+z_6^2+z_7^2)),$$
	where $Z=z_1 Z_1 + \hdots +z_7 Z_7\in\zz$.
	This is used to distinguish isomorphism classes of Lie algebras, see for instance \cite{LO}. 
\end{rem}

\smallskip

{\bf Case (16,8)} In this situation, the linear operator $j$ can be written in terms of the maps given in Equation (\ref{j87}). For $Z=z_1 Z_1 + \hdots +z_8 Z_8\in\zz$, there is a basis $V_1,\ldots, V_{16}$ of $\vv$ such that  
$$j_Z=\left(\begin{matrix}
	j'_{z_1 Z_1 + \hdots +z_7 Z_7}  &   -z_8 Id \\
	z_8 Id         &   -j'_{z_1 Z_1 + \hdots +z_7 Z_7}  
\end{matrix}\right)$$
consists of $8\times 8$ blocks, where $j'$ is given in (\ref{j87}) and $Id$ is the $8\times 8$ identity matrix. So the spaces $\vv_1=span\{V_1,\ldots,V_8\}$ and $\vv_2=span\{V_9,\ldots,V_{16}\}$ are invariant under $j_{Z_1}$ and $j_{Z_2}$ (they are $\mathrm{Cl}(7)$-submodules). 

With this setting we construct the associated Lie algebras $\nn_r$ and we write down the Coniditon C2 for the existence of closed 2-forms of type II. We  also obtain the non-existence of magnetic fields of type II by solving the corresponding linear system.

\begin{prop}\label{nonexistnr} Every closed magnetic field on a Lie algebra $\nn_r$  as above is of type I, for any $r$. 
\end{prop}

\begin{rem}
	Notice that the Lie algebras $\nn_r$ from the proposition are non-singular but also, almost non-singular. In any case, magnetic fields are of type I. 
\end{rem}

Notice that for $r=1$ the Lie algebra $\nn_r$ coincides with the $H$-type Lie algebra. In this case it was proved previously that the corresponding system only admits the trivial solution, that is, the  corresponding matrix of the system is non-singular. Since the set of non-singular matrices is open, there exists an open subset containing the matrix for $r=1$ such that all matrices in this open subset are non-singular, so that one obtains the trivial solution whenever $r$ approaches to $r=1$. The  statement in the proposition \ref{nonexistnr} holds for this particular example. But the reasoning in this paragraph is independent of the decomposition $\vv_1 \oplus \vv_2$ and it still  holds for the general construction of $\nn_r$ explained previously.

 
 
After the results in the paper a natural question remains still open:

\smallskip

{\em Open question: Are the $H$-type Lie algebras in Theorem \ref{thm2} the only fat Lie algebras admitting closed 2-forms of type II? }

\end{document}